\documentclass[12pt]{amsart}

\usepackage{geometry}
\newtheorem{thm}{Theorem}[section]
\newtheorem{lemma}[thm]{Lemma}

\newtheorem{cor}[thm]{Corollary}
\newtheorem{prop}[thm]{Proposition}
\newtheorem{conj}[thm]{Conjecture}

\theoremstyle{definition}

\newtheorem{example}[thm]{Example}

\newtheorem{defn}[thm]{Definition}

\numberwithin{equation}{section}

\usepackage{amssymb,amsmath,graphicx,tikz}
\usepackage{xcolor}

\usepackage{algpseudocode}
\usepackage{algorithm}

\newcommand{\ring}[1]{\ensuremath{\mathbb{#1}}}

\newcommand\ZZ{\ring{Z}}

\renewcommand\iff{\Leftrightarrow}

\newcommand\set[1]{\left\{#1\right\}}

\DeclareMathOperator\coker{coker}

\DeclareMathOperator{\diag}{diag}
\DeclareMathOperator{\Span}{span}

\usepackage{todonotes}

\title[Smith Normal Forms of Graphical Hermite Simplices]{Smith Normal Forms of Graphical Hermite Simplices}

\author[B.~Braun]{Benjamin Braun}
\address{Mathematics Department\\University of Kentucky\\Lexington, KY 40506}
\email{benjamin.braun@uky.edu}

\author[A.~Park]{Antwon Park}
\address{Mathematics Department\\University of Kentucky\\Lexington, KY 40506}
\email{antwon.park@uky.edu}

\date{22 June 2026}

\begin{document}

\begin{abstract}
We introduce the family of directed graphical simplices and study the Smith normal forms of their matrices of vertex vectors, which is equivalent to studying the group structure of the cokernels for these matrices.
Our motivation is to study the behavior of lattice simplices subject to small lattice perturbations of their vertices.
In this case, a directed graphical simplex is a perturbation of a rectangular simplex, i.e., a simplex defined by a diagonal matrix and the origin, with the perturbation controlled by the structure of a directed graph.
We first establish sufficient conditions on the graphs and diagonal entries of these matrices that imply having a single non-unit invariant factor, i.e., a cyclic cokernel.
We then obtain bounds on the invariant factors of the defining matrices related to lengths of paths in the corresponding directed graph.
\end{abstract}

\thanks{The first author was partially supported by US National Science Foundation award DMS-2450299.}

\maketitle

\section{Introduction}
\label{sec:introduction}

The study of geometric and algebraic properties of lattice simplices, including those presented by matrices in Hermite normal form, has been of significant recent interest~\cite{abendschymura,localhstarpaper,LaplacianSimplicesDigraphs,triangulationslecturehall,braundavisanti,BraunDavisReflexive,braundavishanelylanesolus,
	BraunDavisSolusIDP,braunhanely,BraunLiu,braunolsenehrhartlimits,rectsimppaper,hibihermite,selfdualreflexivesimplices,gorensteinproperties,gorensteingivendelta,
	higashitanicounterexamples,
	higashitaniprime,
	higashitanisimplicesmaxdimension,
	universalsimplices,
	SolusNumeralSystems,tsuchiyagorenstein}.
Bruns and Gubeladze~\cite{rectsimppaper} studied algebraic properties of rectangular simplices, i.e., simplices obtained as the convex hull of the zero vector and the columns of a diagonal matrix with positive integer entries.
These are in one sense the simplest matrices in Hermite normal form, yet the arithmetic structure of the integer points in the simplicial cone over a rectangular simplex is complicated.
Various simplices of recent interest can be thought of as small lattice perturbations of a rectangular simplex, for example, lecture hall simplices after a suitable unimodular transformation~\cite[Section 3]{triangulationslecturehall}.

Motivated by this observation, we initiate the study of a family of simplices that we refer to as directed graphical simplices. 
These are simplices obtained as small lattice perturbations of a rectangular simplex, where the perturbation is determined by the structure of an acyclic directed graph.
A special subfamily of these are the graphical Hermite simplices, obtained when the edges of the acyclic directed graph is directed towards the vertex with the larger label.
Our goal in this paper is to study the structure of the cokernel of the matrix of vertices for such a simplex, which is done by studying the Smith normal form of the matrix.
The study of the Smith normal form of an integer matrix, or equivalently of the cokernel, is an important topic both within and outside the context of lattice simplices~\cite{distinguishinggraphsnf,criticalstarscomplete,oaxacaarithstructures,criticalsignedgraphs,criticalstarclique,2025picardjacobian,snfdifferential,stanleysnfcomb,wangstanley}.

Our main results and the structure of this paper are as follows.
In Section~\ref{sec:snf} we review basic properties of Smith and Hermite normal forms and define directed graphical simplices.
In Section~\ref{sec:cyclic_cokernel}, we identify in Theorem~\ref{thm:cyclic_cokernel} properties of both the diagonal values and the directed graph that yield cyclic cokernels, or equivalently, that yield a single non-unit elementary divisor in the Smith normal form.
In Section~\ref{sec:constant_diagonal}, we consider the case of directed graphical simplices with constant diagonal values.
We prove in Theorem~\ref{thm:largest_bound} an upper bound on the largest elementary divisor in this case.

\section{Smith normal form and graphical simplices}
\label{sec:snf}

Let $[n] := \set{1,2,\dots,n}$. We will be working with diagonal matrices, which we denote by
\[
\diag(d_1,d_2,d_3,\dots, d_n) = \begin{bmatrix}
	d_1 & 0 & 0 & \dots & 0 \\
	0 & d_2 & 0 & \dots & 0 \\
	0 & 0 &  d_3 & \dots & 0 \\
	\vdots & \vdots & \vdots & \ddots & \vdots \\
	0 & 0 & 0 & \dots & d_n
\end{bmatrix} \, ,
\]
where for a non-square matrix there might be additional columns or rows of zeros to the right or below the square matrix containing the diagonal entries.

Let $A \in \ZZ^{n \times n}$ be nonzero. 
Let us denote $A[i,j]$ as the entry of $A$ in the $i$-th row and the $j$-th column.
It is known \cite{newmanbook} that there exists matrices $S \in \ZZ^{n\times n}$ and $T \in \ZZ^{n \times n}$ such that the product $SAT$ is the diagonal matrix
\[
SAT = \diag(\alpha_1,\alpha_2,\alpha_3,\dots, \alpha_r, 0, 0 ,\dots, 0) \, ,
\]
where each $\alpha_i$ is a non-negative integer such that $\alpha_i \mid \alpha_{i+1}$ for $1 \leq i < r$. 
This diagonal matrix is called the \emph{Smith normal form} of $A$, which we abbreviate as SNF. 
The entries $\alpha_i$ are unique and are called the \emph{elementary divisors} of $A$. 

The Smith normal form of an integral matrix can be obtained by a sequence of invertible elementary row and column operations, i.e., permuting rows or columns, adding an integer multiple of a row to another row, negating a row, etc.
These operations preserve the determinant and thus we have
\[
\det(A) = \prod_{i=1}^{n} \alpha_i \, .
\]

It is also known \cite{newmanbook} that the elementary divisors of a matrix can be computed as
\begin{equation}\label{eq:elementary_divisors_computation}
	\alpha_i = \frac{\delta_i(A)}{\delta_{i-1}(A)},
\end{equation}
where $\delta_0(A) = 1$ and $\delta_i(A)$ is the greatest common divisor of the determinants of all $i\times i$ submatrices of $A$.

\begin{example}
	Let 
	\[A = \begin{bmatrix}
		3 & 0 & 0 \\
		0 & 6 & 1 \\
		0 & 0 & 9
	\end{bmatrix}.\]
	We can see that $\delta_1(A) = 1$, i.e., the $\gcd$ of the entries is 1. 
	Thus we have $\alpha_1 = 1$. 
	For $\delta_2(A)$, one can show that each $2 \times 2$ minor of $A$ has determinant divisible by $3$ with the minor obtained by deleting the third row and second column having determinant exactly 3. 
	Thus, $\delta_2(A) = 3$, and so $\alpha_2 = \dfrac{\delta_2(A)}{\delta_1(A)} = 3$. Since $\delta_3(A) = \det A = 162$, we have $\alpha_3 = 54$. 
	Thus, the SNF of $A$ is $\diag(1,3,54)$.
\end{example}

In this work, our main object of study is the SNF of the following class of matrices, which have as input a positive integer vector of length $n$ and a finite simple acyclic directed graph, i.e., a DAG, on $n$ vertices. We will assume that all graphs are simple acyclic directed graphs.

\begin{defn}\label{defn:DGS}
	Given a DAG $G$ on $[n]$ and a vector $\vec{d} \in \ZZ_{\geq 1}^n$, we define the \emph{directed graphical simplex}  $A_{\vec{d}, G}$ to be the matrix 
	\[
	A_{\vec{d}, G} := \diag(d_1,d_2,\dots,d_n) + \sum_{(i,j) \in E(G)} E_{i,j} \in \ZZ_{\geq 0}^{n\times n},
	\]
	where $(i,j)$ is a directed edge, originating at vertex $i$ and connecting to vertex $j$, and $E_{i,j}$ is the $n \times n$ standard basis matrix such that the only nonzero entry is $E_{i,j}[i,j] = 1$.
\end{defn}

\begin{example}\label{ex:GHS_on_3}
	Let $\vec{d} = (2,5,3)$ and consider the following directed graphs on $[3]$:
	\[\begin{array}{ccc}
		G_1 = \begin{tikzpicture}
			\node[circle, draw] (A) at (0,0) {1};
			\node[circle, draw] (B) at (1,0) {2};
			\node[circle, draw] (C) at (2,0) {3};

		\end{tikzpicture} & \hspace{2cm} & A_{\vec{d}, G_1} = \begin{bmatrix}
			2 & 0 & 0\\
			0 & 5 & 0\\
			0 & 0 & 3
		\end{bmatrix}\\[1cm]
		
		G_2 = \begin{tikzpicture}
			\node[circle, draw] (A) at (0,0) {1};
			\node[circle, draw] (B) at (1,0) {2};
			\node[circle, draw] (C) at (2,0) {3};
			
			\draw[thick, ->] (B) -- (C);
			
		\end{tikzpicture} & \hspace{2cm} & A_{\vec{d}, G_2} = \begin{bmatrix}
			2 & 0 & 0\\
			0 & 5 & 1\\
			0 & 0 & 3
		\end{bmatrix}\\[1cm]
		
		G_3 = \begin{tikzpicture}
			\node[circle, draw] (A) at (0,0) {1};
			\node[circle, draw] (B) at (1,0) {2};
			\node[circle, draw] (C) at (2,0) {3};
			
			\draw[thick, ->] (A) -- (B);
			\draw[thick, ->] (B) -- (C);
			
		\end{tikzpicture} & \hspace{2cm} & A_{\vec{d}, G_3} = \begin{bmatrix}
			2 & 1 & 0\\
			0 & 5 & 1\\
			0 & 0 & 3
		\end{bmatrix}\\[1cm]
		
		G_4 = \begin{tikzpicture}
			\node[circle, draw] (A) at (0,0) {2};
			\node[circle, draw] (B) at (1,0) {1};
			\node[circle, draw] (C) at (2,0) {3};
			
			\draw[thick, ->] (A) -- (B);
			\draw[thick, ->] (B) -- (C);
			
		\end{tikzpicture} & \hspace{2cm} & A_{\vec{d}, G_4} = \begin{bmatrix}
			2 & 0 & 1\\
			1 & 5 & 0\\
			0 & 0 & 3
		\end{bmatrix}\\[1cm]
		
		G_5 = \begin{tikzpicture}
			\node[circle, draw] (A) at (1,0) {1};
			\node[circle, draw] (B) at (0,0) {2};
			\node[circle, draw] (C) at (2,0) {3};
			
			\draw[thick, ->] (A) -- (B);
			\draw[thick, ->] (A) -- (C);
			
		\end{tikzpicture} & \hspace{2cm} & A_{\vec{d}, G_5} = \begin{bmatrix}
			2 & 1 & 1\\
			0 & 5 & 0\\
			0 & 0 & 3
		\end{bmatrix}\\[1cm]

	\end{array}\]
\end{example}

Note that in Example \ref{ex:GHS_on_3} above, $A_{\vec{d}, G_4}$ is the only matrix that is not upper triangular. This is because the edge between 1 and 2 is directed towards the smaller labeled vertex. In fact, $A_{\vec{d}, G}$ is upper triangular if and only if each edge in $G$ is directed towards the vertex with the larger label. We say that $G$ is an \emph{increasing DAG} if each edge is oriented in this way. Increasing DAGs are closely related to the topological ordering of a DAG, and such a construction is guaranteed to be acyclic. 

\begin{defn}\label{defn:GHS}
	Let $G$ be an increasing DAG on $[n]$.
	We call $A_{\vec{d}, G}$ a \emph{graphical Hermite simplex}.
\end{defn}

We will briefly explain the motivation for the name graphical Hermite simplex.
If $G$ is an increasing DAG, then $A_{\vec{d}, G}$ is automatically in \emph{Hermite normal form}, where $B \in \ZZ^{n\times n}$ is in Hermite normal form if
\begin{enumerate}
	\item $B[i,j] \in\ZZ_{\geq 0}$ for all $i$ and $j$,
	\item $B$ is upper triangular, and
	\item each column of $B$ has a unique maximum entry located on the main diagonal of $B$.
\end{enumerate} 

We obtain a lattice simplex as the convex hull of the zero vector and the columns of $A_{\vec{d}, G}$, and it is the algebraic and geometric properties of these simplices that we are interested in.
It is known that every lattice simplex is uniquely represented as the convex hull of the zero vector and a square matrix in Hermite normal form~\cite[Theorem 4.1~and~Corollary 4.2a]{schrijver}.
Many important invariants of these simplices, such as the Ehrhart $h^*$-polynomial and the integer decomposition property, are controlled by the integer points of the fundamental parallelepiped of the simplex; these integer points have a finite abelian group structure determined by the SNF of $A_{\vec{d} ,G}$.

\begin{defn}\label{defn:FP}
	The \emph{fundamental parallelepiped} of a directed graphical simplex $A_{\vec{d}, G}$ is the set
	\[\mathcal P(A_{\vec{d}, G}) = \set{\left[\begin{array}{c|c}
			1 & \vec{1} \hspace{0.3mm}\phantom{}^T\\
			\hline\\[-1em]
			\vec{0} & A_{\vec{d}, G}
		\end{array}\right]\vec{x}  : \vec{x} \in [0,1)^{n+1}} = \set{\sum_{i=0}^n \lambda_i \begin{bmatrix}
			1 \\ \vec{v_i}
		\end{bmatrix}  : \lambda_i \in [0,1) },\]
	where $\vec{v_0}$ is the zero vector and $\vec{v_i}$ is the $i$-th column of $A_{\vec{d}, G}$.
\end{defn} 
The following proposition is well-known~\cite{stanleysnfcomb}.

\begin{prop}\label{thm:SNF_group}
	Using the notation as in Definition \ref{defn:FP}, the integer points of $\mathcal{P}(A_{\vec{d}, G})$ form a group under addition given by the cokernel of $\mathcal P(A_{\vec{d}, G})$:
	\[\coker(A_{\vec{d}, G}) = \ZZ^{n+1} / \Span_\ZZ\left(\vec{e_1}, \begin{bmatrix}
		1 \\ \vec{v_1}
	\end{bmatrix}, \dots, \begin{bmatrix}
		1 \\ \vec{v_n}
	\end{bmatrix}\right) \cong \ZZ/\alpha_1\ZZ \oplus \ZZ/\alpha_2\ZZ \oplus \dots \oplus \ZZ/\alpha_n \ZZ,\]
	where $\alpha_i$ are the elementary divisors of $A_{\vec{d}, G}$.
\end{prop}

Despite our initial motivation of studying generalizations of simplices presented by matrices in Hermite normal form, our results will focus on directed graphical simplices, i.e., we will not require $G$ to be an increasing DAG. However, every directed graphical simplex has the same SNF as some graphical Hermite simplex. In particular, the SNF of a directed graphical simplex is preserved by a relabeling of the vertices of $G$ with a corresponding reindexing of the diagonal vector. 
We denote by $S_n$ the set of permutations on $n$ elements.

\begin{prop}\label{thm:relabeling}
	Let $G$ be a DAG on $[n]$, let $\vec{d} \in \ZZ^n$, and let $T \in S_n$. Then let $T(G)$ denote the directed graph obtained by replacing the label of each vertex $v$ with $T(v)$ and let
	\[
	T(\vec{d}) = (d_{T(1)}, d_{T(2)}, \dots ,d_{T(n)})\, .
	\] 
	Then $A_{\vec{d}, G}$ and $A_{T(\vec{d}), T(G)}$ have the same SNF. 
\end{prop}
\begin{proof}
	Since every permutation can be written as a product of transpositions, it suffices to show that for any DAG $G$ on $[n]$ and any single transposition $T = (v\ w) \in S_n$, the matrices $A_{\vec{d}, G}$ and $A_{T(\vec{d}), T(G)}$ have the same SNF. 
	This can be done by showing that $A_{T(\vec{d}), T(G)}$ can be obtained from $A_{\vec{d}, G}$ using elementary row and column operations.
	Let $A'$ be the matrix obtained by swapping rows $v$ and $w$ of $A_{\vec{d}, G}$ followed by swapping columns $v$ and $w$. 
	It is a straightforward exercise to verify that $A' = A_{T(\vec{d}), T(G)}$, and thus these matrices have the same SNF.
\end{proof}

Proposition \ref{thm:relabeling} allows us to assume without loss of generality that $G$ is an increasing DAG, and then show any result also extends to any DAG isomorphic to $G$ (with an appropriate reordering of $\vec{d}$).

Another helpful tool for studying directed graphical simplices is to use indicator variables for the non-diagonal entries as follows. We define indicator variables $x_{i,j} \in \set{0,1}$ for $i\neq j$, where $x_{i,j}$ evaluates to 1 if and only if $(i,j) \in E(G)$. 
Thus, we can write
\[
A_{\vec{d}, G}[i,j] = \begin{cases}
	d_i & \text{if } i = j,\\
	x_{i,j} & \text{otherwise}.
\end{cases}\]
We use monomials formed from the indicator variables $x_{i,j}$ to identify paths in $G$. 

\begin{defn}
	Given a DAG $G$, a \emph{path} $P \in G$ is a sequence of vertices $(v_1 \to v_2 \to v_3 \to \dots \to v_k)$ such that $(v_i, v_{i+1}) \in E(G)$ for all $1 \leq i \leq k-1$. We denote $\ell(P)$ to be the \emph{length} of $P$, which is the number of edges traversed in $P$.
\end{defn}

\begin{example}
	Let $G$ be a DAG on $[5]$. Then we have the following equivalences:
	\begin{align*}
		x_{1,2}x_{2,3}x_{3,4} = 1 \qquad &\iff \qquad (1\to 2\to 3\to 4) \in G\\
		x_{1,3}x_{3,5} =1 \qquad &\iff \qquad (1 \to 3 \to 5) \in G\\
		x_{2,5} = 1\qquad &\iff \qquad (2 \to 5) \in G
	\end{align*}
\end{example} 

Although we can define these indicator variables for any DAG $G$, we would like to restrict $G$ to be an increasing DAG. If $G$ is an increasing DAG, we can write $A_{\vec{d}, G}$ as
\[A_{\vec{d}, G} = \diag(d_1,d_2,\dots, d_n) + \sum_{1 \leq i < j \leq n} x_{i,j}E_{i,j} = \begin{bmatrix} 
	d_1 & x_{1,2} & x_{1,3} & \dots & x_{1,n}\\
	0 & d_2 & x_{2,3} & \dots & x_{2,n}\\
	0 & 0 & d_3 & \dots & x_{3,n}\\
	\vdots & \vdots & \vdots & \ddots & \vdots\\
	0 & 0 & 0 & \dots & d_n	
\end{bmatrix}.\]
Restricting $G$ to be an increasing DAG forces $A_{\vec{d}, G}$ to be upper triangular, and this will simplify some of the future proofs. 
Using these indicator variables, we next prove a result regarding the minors of $A_{\vec{d}, G}$. 
We first define the notion of gaps of a path.
\begin{defn}
	Given an increasing path $P = (v_1\to v_2 \to \dots \to v_k)$ in an increasing DAG $G$ on $[n]$, we say that $u \in [n]$ is a \emph{gap} of $P$ if $v_1 < u < v_k$ and $u\notin P$.
	For $\vec{d} = (d_1,d_2,\dots, d_n) \in \ZZ^n$, we define the function $g(P)$ as 
	\[
	g(P) := \prod_{i \in S} d_i \, ,
	\]
	where $S \subseteq [n]$ is the set of gaps of $P$. 
	We set $g(P) := 1$ when $P$ contains no gaps.
\end{defn}

\begin{lemma}\label{thm:minor_det}
	Let $G$ be an increasing DAG on $[n]$ and let $x_{i,j}$ be the indicator variable for the edge $(i,j) \in G$. Also, let $\vec{d} = (d_1, d_2, \dots, d_n) \in \ZZ_{\geq 0}^n$.
	If $M_{a,b}$ is the $b \times b$ submatrix of $A_{\vec{d}, G}$ given by
	\[
	M_{a,b} = \begin{bmatrix}
		x_{a,a+1} & x_{a,a+2} & \dots & x_{a,a+b-1} & x_{a,a+b} \\
		d_{a+1} & x_{a+1, a+2} & \dots & x_{a+1,a+b-1} & x_{a+1,a+b} \\
		0 & d_{a+2} & \dots & x_{a+2,a+b-1} & x_{a+2,a+b} \\
		\vdots & \vdots & \ddots & \vdots & \vdots \\
		0 & 0 & \dots & d_{a+b-1} & x_{a+b-1,a+b} \\
	\end{bmatrix},
	\]
	then the determinant of $M_{a,b}$ is equal to
	\[
	\sum_{P = (a \to \dots \to a+b) \in G}   (-1)^{b-\ell(P)} g(P)  \, ,
	\]
	where $\ell(P)$ is the length of $P$ and the sum is equal to $0$ if no such paths exist.
\end{lemma}
\begin{proof}
	We will proceed via induction on $b$. Let $b = 1$. Recall that $\det M_{a,1} = x_{a,a+1}$ evaluates to 1 if and only if the path $(a \to a+1) \in G$. 
	As this path has no gaps and has length $1$, we have the desired outcome of
	\[\det M_{a,1} = \sum_{P = (a \to a+1) \in G} (-1)^{1-\ell(P)} g(P) = \begin{cases} 1 & (a \to a+1) \in G\\ 0 & \text{otherwise} \end{cases}.\]
	
	Assume the inductive hypothesis for $k \leq b-1$.
	We perform cofactor expansion on $M_{a,b}$ along the first row. 
	Note that deleting column $a+i$ for $i < b$ exposes a row containing only indicator variables of the form $x_{a+i,j}$. 
	In addition, in the cofactor expansion, the columns to the left of the deleted column will have entries of $d_j$ along the main diagonal, whereas the columns to the right of the deleted column will form the submatrix $M_{a+i, b-i}$. 
	Thus in the cofactor expansion, the submatrix obtained by deleting the first row and $i$-th column, for $i < k$, has determinant 
	\[ 
	\det M_{a+i, b-i} \prod_{j = a+1}^{a+i-1} d_j  \, .
	\]
	When $i = b$, then we have that the corresponding submatrix obtained is an upper triangular matrix with entries of $d_j$ along the diagonal, hence that submatrix's determinant is $\prod_{j = a+1}^{a+b-1} d_j$. 
	So, from cofactor expansion and our inductive hypothesis, we have 
	\begin{align*}
		\det M_{a,b} &=  (-1)^{k-1} x_{a,a+b} \prod_{j = a+1}^{a+b-1} d_j + \sum_{i=1}^{b-1} (-1)^{i-1}  x_{a,a+i} \det M_{a+i, b-i} \prod_{j = a+1}^{a+i-1} d_j\\
		&= (-1)^{b-1} x_{a,a+b} \prod_{j = a+1}^{a+b-1} d_j + \sum_{i=1}^{b-1} (-1)^{i-1}  x_{a,a+i} \left(\sum_{P = (a+i \to \dots \to a+b) \in G} (-1)^{b-i-\ell(P)} g(P)\right) \prod_{j = a+1}^{a+i-1} d_j.
	\end{align*}
	
	Note that for $i < b$, all paths $P = (a+i \to \dots \to a+b)$ can be extended to a path from $a$ to $a+b$ by adding the edge $(a, a+i)$. 
	Note that this resulting path keeps all of its previous gaps in addition to $\set{a+1,a+2,\dots,a+i-1}$. 
	Also, the edge $(a,a+b)$ can be interpreted as a path with gaps $\set{a+1,a+2,\dots,a+b-1}$. 
	Therefore, if we instead describe $P$ as the path $(a \to a+i \to \dots \to a+b)$, then we can write
	\begin{align*}
		\det M_{a,b} &=  \sum_{P = (a \to a+b) \in G }(-1)^{b-\ell(P)} g(P) + \sum_{i=1}^{b-1} \sum_{P = (a \to a+i \to \dots \to a+b) \in G} (-1)^{b-\ell(P)} g(P)\\
		&= \sum_{P = (a \to \dots \to  a+b) \in G }(-1)^{b-\ell(P)} g(P).
	\end{align*}
\end{proof}

The submatrix $M_{a,b}$ will be used several times throughout the paper, as paths in our graph $G$ will tell us some information regarding the determinant of $M_{a,b}$. In particular,  the determinant of $M_{a,b}$ is nonzero only if there exists a path from $a$ to $a+b$, and if we add additional conditions on the paths of $G$, then we can often make conclusions regarding the elementary divisors of $A_{\vec{d}, G}$. 

\section{Cyclic Cokernel}
\label{sec:cyclic_cokernel}

In this section, we establish conditions on $\vec{d}$ and $G$ that produce a directed graphical simplex with a cyclic cokernel.  

\begin{thm}\label{thm:cyclic_cokernel}
	Let $G$ be a DAG on $[n]$ and $\vec{d} \in \ZZ^n_{\geq 1}$. If one of the following hold:
	\begin{enumerate}
		\item The entries of $\vec{d}$ are pairwise coprime and $G$ is arbitrary,
		\item $G$ is a path of length $n-1$ and $\vec{d}$ is arbitrary,
		\item $G$ contains a path of length $n-1$ and $\vec{d} = (m,m,\dots,m) \in \ZZ_{\geq 1}^n$ ,
	\end{enumerate}
	then $A_{\vec{d}, G}$ has a cyclic cokernel.
\end{thm}

The rest of this section contains propositions yielding the implications stated in Theorem \ref{thm:cyclic_cokernel}. 
Note that by the divisibility of the elementary divisors of $A_{\vec{d}, G}$, it suffices to show that $\delta_{n-1}(A_{\vec{d}, G}) = 1$ to show that $A_{\vec{d}, G}$ has a cyclic cokernel.

\begin{prop}\label{thm:pairwise_coprime}
	Suppose that the entries of $\vec{d} \in \ZZ_{\geq 1}^n$ are pairwise coprime. 
	Then for any DAG $G$, the matrix $A_{\vec{d}, G}$ has only one nonunit elementary divisor.
\end{prop}
\begin{proof}
	By Proposition \ref{thm:relabeling}, we can assume that $G$ is an increasing DAG, thus $A_{\vec{d}, G}$ is upper triangular.
	Consider the $(n-1) \times (n-1)$ submatrix of $A_{\vec{d}, G}$ obtained by removing the $j$-th row and column for some $j \in [n]$. This submatrix is upper triangular, thus we have the determinant of this submatrix is
	\[\frac{\prod_{i=1}^n d_i}{d_j}.\]
	Since this is true for all $j \in [n]$, we have that the greatest common divisor of the determinants of the $(n-1) \times (n-1)$ submatrices must be 1. 
	Thus, $\delta_{n-1}(A_{\vec{d}, G}) = 1$.
\end{proof}

In Example \ref{ex:GHS_on_3}, we have that $\vec{d} = (2,5,3)$ satisfies the conditions in the above Proposition. Therefore, for any DAG $G$ on $[3]$, we have that the SNF of $A_{(2,5,3), G}$ is $\diag(1,1,30)$.

\begin{prop}\label{thm:path_SNF}
	Let $G$ be a directed path of length $n-1$ on $[n]$. 
	Then $A_{\vec{d}, G}$ has only one nonunit elementary divisor for all $\vec{d} \in \ZZ_{\geq 1}^n$.
\end{prop}
\begin{proof}
	By Proposition \ref{thm:relabeling}, we can assume that $G$ is the path $(1 \to 2 \to \dots \to n)$.
	From Lemma \ref{thm:minor_det}, since there is only one path from 1 to $n$ and this path does not contain any gaps, we have $\det M_{1,n-1} = 1.$
	Therefore, there exists a $(n-1) \times (n-1)$ submatrix of $A_{\vec{d}, G}$ with determinant 1, hence $\delta_{n-1}(A_{\vec{d}, G}) = 1$.
\end{proof}

Note that in the previous proposition, $G$ must be exactly a path of length $n-1$. 
It is possible for the presence of additional edges beyond the path edges to cause $\alpha_{n-1} > 1$, as demonstrated in the following example.

\begin{example}\label{ex:full_spine_not_cyclic_cokernel}
	Let $\vec{d} = (6,9,12)$ and let $G$ be the DAG on [3] given by
	\[\begin{tikzpicture}
		\node[circle, draw] (1) at (0,0) {1};
		\node[circle, draw] (2) at (1,1) {2};
		\node[circle, draw] (3) at (2,0) {3};
		
		\draw[thick, ->] (1) -- (2);
		\draw[thick, ->] (2) -- (3);
		\draw[thick, ->] (1) -- (3);
		
	\end{tikzpicture}\] 
	Observe that $G$ contains the path $(1\to 2 \to 3)$, yet the SNF of $A_{\vec{d}, G}$ is $\diag(1,2,324)$. 
	Thus, we have a graph containing a path of length $2$, yet the cokernel of $A_{\vec{d}, G}$ is not cyclic.
\end{example}

In general, when the entries of $\vec{d}$ are arbitrary, knowing how each edge in $G$ will affect the SNF of $A_{\vec{d}, G}$ is more complicated. 
Example \ref{ex:full_spine_not_cyclic_cokernel} demonstrates that even with $\gcd (\vec{d})$ being nonunit, it is nontrivial to locate some graphical pattern that would imply $A_{\vec{d}, G}$ having a cyclic cokernel.
If we restrict to a constant diagonal, where $\vec{d} = (m,m,\dots,m) \in \ZZ_{\geq 1}^n$, then we observe that paths have an impact on the SNF of $A_{\vec{d}, G}$.

\begin{cor}\label{thm:full_spine}
	Let $G$ be a DAG on $[n]$ containing a path of length $n-1$. 
	If $\vec{d} = (m,m,\dots,m) \in \ZZ_{\geq 1}^n$, then the cokernel of $A_{\vec{d}, G}$ is cyclic.
\end{cor}
\begin{proof}
	This is an immediate consequence of Theorem \ref{thm:unique_largest_eq}, as there is only one path of length $n-1$.
\end{proof}

Note that Proposition \ref{thm:path_SNF} requires $G$ to be exactly a path of length $n-1$, whereas Corollary \ref{thm:full_spine} only requires $G$ to contain a path of length $n-1$. 

In summary, when $\vec{d}$ is pairwise coprime, the SNF is fully determined and the graph structure has no impact on the SNF of $A_{\vec{d}, G}$.
For arbitrary $\vec{d}$, the effect $G$ has on the SNF of $A_{\vec{d}, G}$ is obscured by the number theory of $\vec{d}$. And from Corollary \ref{thm:full_spine}, we see that when $\vec{d}$ is constant, the graph structure has a noticeable impact on the SNF. So, we will next focus on this case.

\section{Constant Diagonal}
\label{sec:constant_diagonal}

Recall we assume all graphs $G$ are simple acyclic directed graphs (abbr. DAG) on vertex set $[n]$.
As we are focusing on the case when our diagonal vector is constant, we write $\vec{m} \in \ZZ_{\geq 0}^n$ for the vector where each entry is a fixed integer $m$.

We first observe that in general it is not necessarily the case that the elementary divisors of $A_{\vec{m}, G}$, where $G$ is a disjoint union of directed graphs, are the same elementary divisors as the disjoint union of the elementary divisors of $A_{\vec{m}, H}$ where $H$ is a connected component of $G$. 
(However, it is straightforward to show that this is true when we restrict $m$ to be prime.)
The following family of graphs provide concrete examples of how issues can arise when $m$ is composite.

\begin{example}
	Let $C_i$ be the graph on $[i+3]$ such that it contains the path $(2 \to 3 \to \dots \to i+2)$ and the edges $(1,j)$ and $(k, i+3)$ for $2 < j \leq i+2$ and $2 \leq k < i+2$. 
	
	\[\begin{array}{ccc}
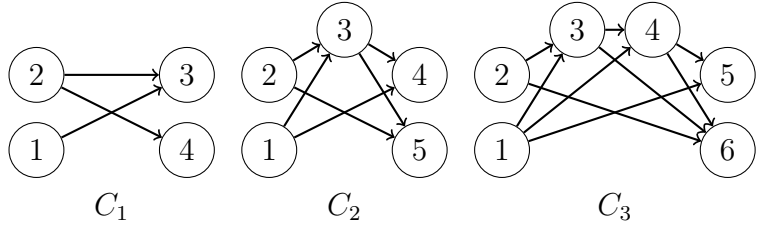
		\begin{tikzpicture}
			\node[circle, draw] (A) at (0,0) {2};
			\node[circle, draw] (B) at (0,-1) {1};
			\node[circle, draw] (C) at (2,0) {3};
			\node[circle, draw] (D) at (2,-1) {4};
			
			\draw[thick,->] (A) -- (C);
			\draw[thick,->] (A) -- (D);
			\draw[thick,->] (B) -- (C);
		\end{tikzpicture}
		&
		\begin{tikzpicture}
			\node[circle, draw] (A) at (0,0) {2};
			\node[circle, draw] (B) at (0,-1) {1};
			\node[circle, draw] (C) at (1, 0.6) {3};
			\node[circle, draw] (D) at (2,0) {4};
			\node[circle, draw] (E) at (2,-1) {5};

			\draw[thick,->] (A) -- (C); 
			\draw[thick,->] (B) -- (C); 
			\draw[thick,->] (C) -- (E); 
			\draw[thick,->] (C) -- (D);
			\draw[thick,->] (A) -- (E);
			\draw[thick,->] (B) -- (D);  
		\end{tikzpicture}
		& 
		\begin{tikzpicture}
			\node[circle, draw] (A) at (0,0) {2};
			\node[circle, draw] (B) at (0,-1) {1};
			\node[circle, draw] (C) at (1, 0.6) {3};
			\node[circle, draw] (D) at (2,0.6) {4};
			\node[circle, draw] (E) at (3,0) {5};
			\node[circle, draw] (F) at (3,-1) {6};

			\draw[thick,->] (A) -- (C); 
			\draw[thick,->] (B) -- (C); 
			\draw[thick,->] (B) -- (D); 
			\draw[thick,->] (B) -- (E);
			\draw[thick,->] (C) -- (D);
			\draw[thick,->] (D) -- (E);  
			\draw[thick,->] (A) -- (F);  
			\draw[thick,->] (C) -- (F);  
			\draw[thick,->] (D) -- (F);  
			
		\end{tikzpicture}\\
		C_1 & C_2 & C_3
	\end{array}\]
	It can be shown that the nonunit divisors of $A_{\vec{m}, C_i}$ are $\gcd(i, m), \tfrac{m^2}{\gcd(i,m)},$ and $ m^{i+1}$. 
	Observe we can have any positive integer appear as an elementary divisor of $A_{m, C_i}$ by manipulating $\gcd(i,m)$.
	For example, if $m = 6$, then the SNF of $A_{\vec{6}, C_2}$ is $\diag(1,1,2, 18, 216)$ and the SNF of $A_{\vec{6}, C_3}$ is $\diag(1,1,1,1,3,12, 1296)$.
	We therefore have that $2$ is an elementary divisor of $A_{\vec{6}, C_2}$ and $3$ is an elementary divisor of $A_{\vec{6}, C_3}$. 
	However, it is impossible for both $2$ and $3$ to be elementary divisors of the same matrix. 
	Therefore, the elementary divisors of $A_{\vec{6}, C_2 \amalg C_3}$ are not the same as the disjoint union of the elementary divisors of $A_{\vec{6}, C_2}$ and $A_{\vec{6}, C_3}$. 
\end{example}

We shall next investigate the influence of paths in $G$ on the SNF of $A_{\vec{m} , G}$. 
In particular, we shall see that the length of the longest path in $G$ controls the largest elementary divisor. 
We will show this by establishing a lower bound on $\delta_{n-1}(A_{\vec{m} , G})$. 
We begin by showing that we can edit the statement of Lemma \ref{thm:minor_det} to the following in the case of a constant diagonal.

\begin{lemma}\label{thm:constant_diag_minors}
	Let $G$ be an increasing DAG and let $M_{a,b}$ refer to the submatrix of $A_{\vec{m}, G}$ obtained by the procedure given in Lemma \ref{thm:minor_det}. Then
	\[\det M_{a,b} = \sum_{P = (a \to \dots \to a+b) \in G} (-1)^{b-\ell(P) } m^{b - \ell(P)},\]
	where $\ell(P)$ is the length of $P$.
\end{lemma}
\begin{proof}
	From Lemma \ref{thm:minor_det}, we have 
	\[\det M_{a, b} = \sum_{P = (a \to \dots \to a+b) \in G}   (-1)^{b-\ell(P)} g(P).\]
	Since each diagonal entry is $m$, we have that $g(P)$ is $m$ raised to the power equal to the number of gaps in $P$. 
	We arrive at the desired result, once we recognize that the number of vertices in $P$ and the number of gaps of $P$ must add up to the number of integers in $[a, a+b]$, which is $b+1$.
\end{proof}

The determinant of these submatrices come from the sum over paths in $G$, however, the summand only takes into account the number of vertices in those paths. 
Let us define $f_{c,r}: \ZZ_{\geq 1} \to \ZZ_{\geq 0}$
to be the counting function such that $f_{c,r}(i)$ counts the number of paths $P = (c \to \dots \to r) \in G$ of length $i$.
We define $A_{c,r}$ to be the $(n-1) \times (n-1)$ submatrix of $A_{\vec{m} ,G}$ obtained by deleting column $c$ and row $r$.

\begin{lemma}\label{thm:n-1_minors}
	Let $G$ be an increasing DAG. If $r \neq c$, then the determinant of $A_{c,r}$ is
	\begin{equation}\label{eq:A_cr_det_sum}
		\det A_{c,r} = \sum_{P = (c \to \dots \to r) \in G} (-1)^{r-c-\ell(P)} m^{n-\ell(P)-1}.
	\end{equation}
	
	If we consider $\det A_{c,r}$ as a polynomial in $m$, then we can write
	\begin{equation}\label{eq:A_cr_det_poly}
		\det A_{c,r} = \sum_{i\geq 1} (-1)^{r-c-i} f_{c,r}(i) m^{n-i-1}.
	\end{equation}
\end{lemma}
\begin{proof}	
	
	If $r < c$, then we have $A_{c,r}$ is upper triangular with $A_{c,r}[r,r] = 0$. 
	Therefore, $\det A_{c,r} = 0$. Since each path in $G$ must be an increasing sequence, there are no paths from $c$ to $r$, hence the empty sum is also zero. 
	Therefore, equation \eqref{eq:A_cr_det_sum} holds. 
	
	Next, assume that $r > c$. Thus, we can write 
	\[A_{c,r} = \left[\begin{array}{c|c|c}
		B_{1,1} & B_{1,2} & B_{1,3} \\
		\hline
		0 & M_{c, r-c} & B_{2,3} \\
		\hline
		0 & 0 & B_{3,3}
	\end{array}\right],\]
	where $B_{1,1} \in \ZZ^{(c-1)\times(c-1)}$, $B_{3,3} \in \ZZ^{(n-r)\times (n-r)}$ and $M_{c,r-c}$ is the submatrix as described in Lemma \ref{thm:minor_det}. 
	Note that $B_{1,1}$ and $B_{3,3}$ are upper triangular, and so $\det B_{1,1} = m^{c-1}$ and $\det B_{3,3} = m^{n-r}$.  
	Thus, we have that 
	\begin{align*}
		\det A_{c,r} &= \det B_{1,1} (\det M_{c, r-c} )\det B_{3,3} \\
		&= \sum_{P = (c \to \dots \to r) \in G} (-1)^{r-c-\ell(P)} m^{n-\ell(P)-1}.
	\end{align*}
	
	Note that the summand only depends on the length of $P$, so we can collect terms by their lengths. Since $r\neq c$, then each path from $c$ to $r$ must have length at least 1, thus we arrive at (\ref{eq:A_cr_det_poly}).
\end{proof}

We are now able to prove the following bound on the largest elementary divisor for a directed graphical simplex with constant diagonal.

\begin{thm}\label{thm:largest_bound}
	Let $G$ be a DAG on $[n]$ and let $h$ be the length of the longest path in $G$. 
	Then the largest elemenatry divisor of $A_{\vec{m} ,G}$ is at most $m^{h+1}$.
\end{thm}	
\begin{proof}
	By Proposition \ref{thm:relabeling}, we can assume that $G$ is an increasing DAG.
	Note that $\det A_{r,r} = m^{n-1}$ for all $r$ as $A_{r,r}$ is upper triangular. By Lemma \ref{thm:n-1_minors}, if $\det A_{c,r}$ is nonzero, then from (\ref{eq:A_cr_det_poly}), $\det A_{c,r}$ is a polynomial in $m$ with the smallest nonzero power of $m$ being at least $n-h-1$. Therefore, $m^{n-h-1}$ divides $\det A_{c,r}$ for all $c,r$. Thus $\delta_{n-1} \geq m^{n-h-1}$ and since $\delta_n = \det A_{\vec{m} ,G} = m^n$, we arrive at
	\[\alpha_n = \frac{\delta_n(A_{\vec{m}, G})}{\delta_{n-1}(A_{\vec{m}, G})} \leq m^{h+1}.\]
\end{proof}

As we will later see in the proof of Theorem \ref{thm:unique_largest_eq}, number theoretic considerations involving the number of paths of particular lengths and $m$ determine which elementary divisors appear. 
For example, suppose that there is only one vertex with indegree 0 and only one vertex with outdegree 0, and there are exactly $k$ distinct paths of maximum length $h-1$ in $G$ between these two vertices. 
Then it is possible for $\gcd(k,m)m^{n-h-1}$ to divide $\det A_{c,r}$ for all possible $(n-1)\times(n-1)$ submatrices, hence when $\gcd(k,m) > 1$, then $\alpha_n < m^{h+1}$.
The following example showcases this phenomenon. 

\begin{example}
	Let us define $B_i$ to be the DAG on $[i+2]$ where the edges $(1,k)$ and $(k, i+2)$ are in $B_i$ for $2\leq k \leq i+1$. 
	
	\[\begin{array}{cccc}
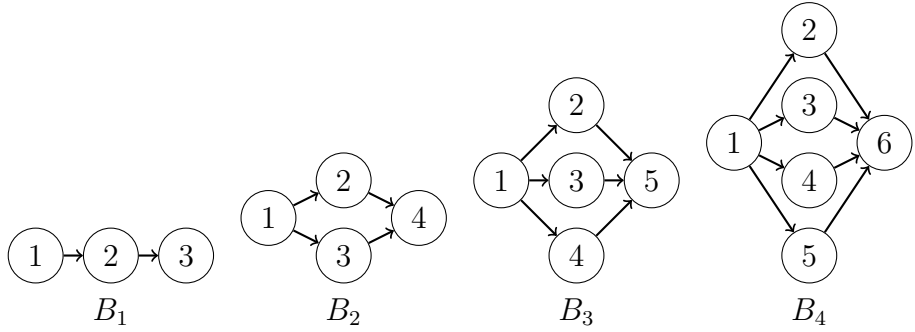
		\begin{tikzpicture}
			\node[circle, draw] (A) at (0,0) {1};
			\node[circle, draw] (B) at (1,0) {2};
			\node[circle, draw] (Z) at (2,0) {3};
			\draw[thick,->] (A) -- (B); 
			\draw[thick,->] (B) -- (Z); 
		\end{tikzpicture}
		&
		\begin{tikzpicture}
			\node[circle, draw] (A) at (0,0) {1};
			\node[circle, draw] (B) at (1,0.5) {2};
			\node[circle, draw] (C) at (1,-0.5) {3};
			\node[circle, draw] (Z) at (2,0) {4};

			\draw[thick,->] (A) -- (B); 
			\draw[thick,->] (A) -- (C); 
			\draw[thick,->] (B) -- (Z); 
			\draw[thick,->] (C) -- (Z);
		\end{tikzpicture}
		& 
		\begin{tikzpicture}
			\node[circle, draw] (A) at (0,0) {1};
			\node[circle, draw] (B) at (1,1) {2};
			\node[circle, draw] (C) at (1,0) {3};
			\node[circle, draw] (D) at (1,-1) {4};
			\node[circle, draw] (Z) at (2,0) {5};

			\draw[thick,->] (A) -- (B); 
			\draw[thick,->] (A) -- (C); 
			\draw[thick,->] (A) -- (D); 
			\draw[thick,->] (B) -- (Z); 
			\draw[thick,->] (C) -- (Z);
			\draw[thick,->] (D) -- (Z);
		\end{tikzpicture}
		& 
		\begin{tikzpicture}
			\node[circle, draw] (A) at (0,0) {1};
			\node[circle, draw] (B) at (1,1.5) {2};
			\node[circle, draw] (C) at (1,0.5) {3};
			\node[circle, draw] (D) at (1,-0.5) {4};
			\node[circle, draw] (E) at (1,-1.5) {5};
			\node[circle, draw] (Z) at (2,0) {6};

			\draw[thick,->] (A) -- (B); 
			\draw[thick,->] (A) -- (C); 
			\draw[thick,->] (A) -- (D); 
			\draw[thick,->] (A) -- (E); 
			\draw[thick,->] (E) -- (Z); 
			\draw[thick,->] (B) -- (Z); 
			\draw[thick,->] (C) -- (Z);
			\draw[thick,->] (D) -- (Z);
		\end{tikzpicture}\\
		B_1 & B_2 & B_3 & B_4
	\end{array}\]
	
	The nonunit elementary divisors for $A_{\vec{6}, B_i}$ are
	\[\begin{array}{c|r}
		B_1 & \set{216}\\
		B_2 & \set{12, 108}\\
		B_3 & \set{6,18,72}\\
		B_4 & \set{6,6, 12,108}\\
		B_5 & \set{6, 6, 6, 6, 216}
	\end{array}\]
	The length of the longest path in $B_i$ is $2$ and there are $i$ total paths of length $2$. 
	From Theorem \ref{thm:largest_bound}, we have that $6^{3} = 216$ is an upper bound for the largest elementary divisor. 
	We see that we achieve this bound only for $i = 1$ and $i=5$, which are the only two integers $i$, for which $B_i$ is listed above, that are coprime with $6$. 
	In fact, we see that the largest elementary divisor of $A_{\vec{6}, B_i}$ is of the form $\frac{6^3}{\gcd(6,i)}$.
\end{example}

One can see that $6^{3}$ is an elementary divisor of $A_{\vec{6}, B_i}$ for all $i$ such that $\gcd(i,6) = 1$. 
In fact, one can check that $m^{3}$ is an elementary divisor of $A_{\vec{m}, B_i}$ for all $i$ such that $\gcd(i,m) = 1$. 
This is not a coincidence. 
This behavior generalizes to all DAGs, as shown by the following theorem.

\begin{thm}\label{thm:unique_largest_eq}
	Let $G$ be a DAG on $[n]$ and let $h$ be the length of the longest path in $G$. 
	Suppose that $\gcd(f_{u,v}(h),m) = 1$ for some $u,v$. 
	Then the largest elementary divisor of $A_{\vec{m}, G}$ is exactly $m^{h+1}$.
\end{thm}

\begin{proof}	
	Assume, by Proposition \ref{thm:relabeling}, that $G$ is an increasing DAG.
	Let $u,v$ be such that $\gcd(f_{u,v}(h), m) = 1$. 
	Recall that $f_{u,v}(h)$ counts the number of paths $(u\to \dots \to v) \in G$ of length $h$. 
	Let us now define $f'_{u,v}(m)$ as the quantity
	\begin{align*}
		f'_{u,v}(m) &= \frac{\det A_{u,v}}{m^{n-h-1}} \\
		&= \sum_{i=1}^{h} (-1)^{v-u-i} f_{u,v}(i) m^{h-i} \\
		&= (-1)^{v-u-h}f_{u,v}(h) + \sum_{i=1}^{h-1} (-1)^{v-u-i} f_{u,v}(i) m^{h-i}.
	\end{align*}
	By our assumption, $\gcd(f_{u,v}(h),m) = 1$ implies that $\gcd(f'_{u,v}(m), m) = 1$. 
	By definition, $\delta_{n-1}(A_{\vec{m}, G})$ must divide $\gcd(\det A_{u,v}, \det A_{1,1})$. 
	Here, we are considering $A_{1,1}$ as this submatrix is upper triangular, which means $\det A_{1,1} = m^{n-1}$. 
	Therefore, we have that $\delta_{n-1}(A_{\vec{m}, G})$ must divide $m^{n-h-1}$. 
	In particular, $\delta_{n-1}(A_{\vec{m}, G}) \leq m^{n-h-1}$. 
	However, we know from Theorem \ref{thm:largest_bound} that $\delta_{n-1}(A_{\vec{m}, G}) \geq m^{n-h-1}$. 
	Therefore, we have $\delta_{n-1}(A_{\vec{m}, G}) = m^{n-h-1}$, hence
	\[
	\alpha_n = \frac{\delta_n(A_{\vec{m}, G})}{\delta_{n-1}(A_{\vec{m}, G})} = m^{h+1}.
	\]
\end{proof}

From our previous discussions, we see that the lengths of paths, specifically the number of paths with maximal length, has an effect on the elementary divisors of $A_{\vec{m}, G}$. 
Inspired by this, we will next explore what happens when you minimize the longest path length. 

\begin{defn}
	A \emph{bipartite DAG} $G$ is a directed graph such that $V(G) = U \amalg V$ and each edge in $G$ is of the form $(u,v)$ where $u \in U$ and $v \in V$. 
\end{defn} 

\begin{example}
	The following are bipartite DAGs with $U = \set{1,2,3}$ and $V = \set{4,5,6}$
	
	\[\begin{array}{ccc}
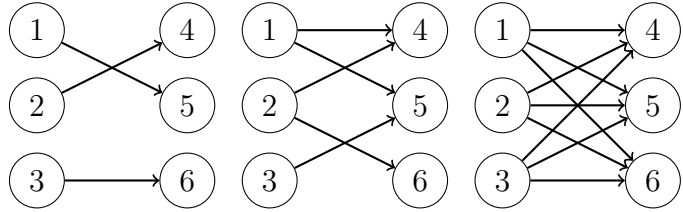
		\begin{tikzpicture}
			\node[circle, draw] (A) at (0, 0) {1};
			\node[circle, draw] (B) at (0,-1) {2};
			\node[circle, draw] (C) at (0,-2) {3};
			\node[circle, draw] (D) at (2, 0) {4};
			\node[circle, draw] (E) at (2,-1) {5};
			\node[circle, draw] (F) at (2,-2) {6};
			
			\draw[thick, ->] (A) -- (E);
			\draw[thick, ->] (B) -- (D);
			\draw[thick, ->] (C) -- (F);
		\end{tikzpicture}
		& 
		\begin{tikzpicture}
			\node[circle, draw] (A) at (0, 0) {1};
			\node[circle, draw] (B) at (0,-1) {2};
			\node[circle, draw] (C) at (0,-2) {3};
			\node[circle, draw] (D) at (2, 0) {4};
			\node[circle, draw] (E) at (2,-1) {5};
			\node[circle, draw] (F) at (2,-2) {6};
			
			\draw[thick, ->] (A) -- (D);
			\draw[thick, ->] (A) -- (E);
			\draw[thick, ->] (B) -- (D);
			\draw[thick, ->] (B) -- (F);
			\draw[thick, ->] (C) -- (E);
		\end{tikzpicture} 
		& 
		\begin{tikzpicture}
			\node[circle, draw] (A) at (0, 0) {1};
			\node[circle, draw] (B) at (0,-1) {2};
			\node[circle, draw] (C) at (0,-2) {3};
			\node[circle, draw] (D) at (2, 0) {4};
			\node[circle, draw] (E) at (2,-1) {5};
			\node[circle, draw] (F) at (2,-2) {6};
			
			\draw[thick, ->] (A) -- (D);
			\draw[thick, ->] (A) -- (E);
			\draw[thick, ->] (A) -- (F);
			\draw[thick, ->] (B) -- (D);
			\draw[thick, ->] (B) -- (E);
			\draw[thick, ->] (B) -- (F);
			\draw[thick, ->] (C) -- (D);
			\draw[thick, ->] (C) -- (E);
			\draw[thick, ->] (C) -- (F);
		\end{tikzpicture}
	\end{array}\]
\end{example}

It is clear by definition that bipartite DAGs have a maximum path length of $1$, and by Proposition \ref{thm:relabeling}, we can assume without loss of generality that $U = \set{1,2,3,4,\dots, k}$ and $V = \set{k+1,\dots, n}$.
From Theorem \ref{thm:largest_bound}, we know that the elementary divisors of $A_{\vec{m}, G}$ are bounded by $m^2$ when $G$ is a bipartite DAG. 
However, the fact that $G$ is bipartite actually allows us to make an even stronger claim, in particular, the SNF of the adjacency matrix of $G$ fully determines the SNF of $A_{\vec{m}, G}$.

\begin{prop}\label{thm:bipartite_elementary_divisors}
	Let G be a bipartite DAG and B be given as 
	\[
	B = \sum_{(i,j)\in E(G)} E_{i,j}. 
	\]
	
	Let $\beta_i$ be the elementary divisors of $B$ and let $\gamma_i := \gcd(m, \beta_i)$. 
	If the rank of $B$ is $r$, then the SNF of $A_{\vec{m}, G}$ is
	\[\diag\left(\gamma_1, \gamma_2, \dots, \gamma_r, m, m, \dots, m, \dfrac{m^2}{\gamma_r}, \dfrac{m^2}{\gamma_{r-1}}, \dots, \dfrac{m^2}{\gamma_{1}}\right).\]
\end{prop}

\begin{proof}
	Without loss of generality, assume that $V(G) = \set{1,2,\dots,k} \amalg \set{k+1,\dots,n}$. 
	Then we have $x_{i,j}$ is nonzero only if $i \leq k$ and $j > k$. Thus, we have that $A_{\vec{m}, G}$ is a block matrix of the form
	\[A_{\vec{m}, G} = \begin{bmatrix}
		mI_k & B \\
		0 & mI_{n-k}
	\end{bmatrix}.\]
	
	The following discussion demonstrates that elementary row operations on $B$ can arise from elementary row operations on $A_{\vec{m} ,G}$ without affecting the other entries in $A_{\vec{m} ,G}$. 
	An identical argument will show the same applies for elementary column operations. 
	Let $i,j \leq k$ be distinct.
	\begin{itemize}
		\item To perform the row operation $R_i \to R_i + R_j$ in $B$, we can perform this operation in $A_{\vec{m}, G}$. In the $mI_k$ block, we are changing the $[i,j]$ position from 0 to $m$. 
		Note that column $i$ of $A_{\vec{m}, G}$ has an $m$ in the $i$-th row and 0 elsewhere. 
		We can then perform the column operation $C_j \to C_j - C_i$ such that $A_{\vec{m}, G}[i,j] = 0$. 
		\item To negate the $i$-th row in $B$, we can negate the $i$-th row in $A_{\vec{m}, G}$ and then negate the $i$-th column in $A_{\vec{m}, G}$.
		\item To swap the $i$-th and $j$-th rows in $B$, we can swap the $i$-th and $j$-th rows in $A_{\vec{m}, G}$. We can then swap the $i$-th and $j$-th columns so that the top left block is unaltered.
	\end{itemize}
	
	Therefore, we know it is possible to arrive at the SNF of $B$ via elementary row and column operations on $B$, so let us perform the corresponding sequence of elementary row and columns operations on $A_{\vec{m}, G}$ until we arrive at 
	\[A_{\vec{p}, G} \to \dots \to \begin{bmatrix}
		mI_k & \diag(\beta_1,\beta_2, \dots, \beta_r, 0, \dots, 0) \\
		0 & mI_{n-k}.
	\end{bmatrix}\]
	
	We can then perform a sequence of row and column permutations of this matrix such that the resulting matrix has $r$ $2\times 2$ blocks along the main diagonal of the form
	\[
	\begin{bmatrix}
		m & \beta_i \\ 0 & m
	\end{bmatrix},
	\]
	followed by a constant $m$ diagonal. 
	Specifically, exchanging columns $2$ and $k+1$ and exchanging rows $2$ and $k+1$ will place a $2\times 2$ block with $\beta_1$ in the upper left corner of $mI_k$ and will place a $2\times 2$ block with $\beta_2$ in the upper left corner of $mI_{n-k}$.
	Iterating this process for consecutive pairs of $\beta$'s will result in the desired form. 	
	
	For each positive $i \leq r$, there exists a sequence of elementary row and column operations (essentially, applying the Euclidean algorithm) such that
	\[\begin{bmatrix}
		m & \beta_i \\ 0 & m
	\end{bmatrix} \to \dots \to  \begin{bmatrix}
		\gamma_i& 0 \\ 0 & \dfrac{m^2}{\gamma_i}
	\end{bmatrix}.\]
	
	Therefore, $A_{\vec{m}, G}$ has the same SNF as 
	\[\diag\left(\gamma_1, \gamma_2, \dots, \gamma_r, m, m, \dots, m, \dfrac{m^2}{\gamma_r}, \dfrac{m^2}{\gamma_{r-1}}, \dots, \dfrac{m^2}{\gamma_{1}}\right).\]
	
	It is straightforward to verify from the divisibility property of the $\beta$'s that each of these terms divides the subsequent term, hence this is the SNF of $A_{\vec{m}, G}$.	
\end{proof}

One special case to consider is when $m$ is prime.
For this case, we denote our diagonal vector as $\vec{p}$ for the prime $p$. 
One reason to consider a prime diagonal is that from Theorem \ref{thm:largest_bound}, the elementary divisors are bounded by $p^2$ and due to the divisibility condition on the elementary divisors, the elementary divisors can only be $1, p$ or $p^2$. 
Then Proposition \ref{thm:bipartite_elementary_divisors} yields the following corollary.

\begin{cor}
	Let $G$ be a bipartite DAG on $[n]$ and $p$ be prime. 
	Then the elementary divisors of $A_{\vec{p}, G}$ are 1, $p$, or $p^2$. 
	In particular, if
	\[B = \sum_{(i,j) \in E(G)} E_{i,j} \in (\ZZ/p\ZZ)^{n \times n},\] 
	then the number of unit elementary divisors is equal to the rank of $B$ over $\ZZ/p\ZZ$. 
	As a consequence, the number of $p^2$ elementary divisors is equal to the rank of $B$ over $\ZZ/p\ZZ$.
\end{cor}
\begin{proof}
	From Corollary~\ref{thm:largest_bound}, we have that the largest elementary divisor is at most $p^2$, so all elementary divisors must be at most $p^2$. 
	Let $\beta_i$ be the elementary divisors of $B$. 
	Note that $\gamma_i = \gcd(p, \beta_i)$ is either $1$ or $p$, and if $\gamma_i = p$ for some $i$, then $\gamma_j = p$ for all $i\leq j \leq r$. 
	Since if $p$ divides $\beta_i$ then that elementary divisor is equal to $0$ mod $p$, we have that $\gamma_i = 1$ for all $i \leq r$, where $r$ is the rank of $B$ over $\ZZ/p\ZZ$, and the rest of the $\gamma_i$ are equal to $p$.
	Thus, from Proposition~\ref{thm:bipartite_elementary_divisors}, we have that the SNF of $A_{\vec{p}, G}$ is
	\[\diag(\underbrace{1,1,1,\dots,1}_{r \text{ copies}}, \underbrace{p,p,\dots,p}_{n-2r \text{ copies}},\underbrace{p^2,p^2,\dots,p^2}_{r \text{ copies}}).\]
\end{proof}

In the proof above, $B$ can also be described as the directed adjaceny matrix of $G$. Thus, we can then claim there are $n-r$ nonunit elementary divisors of $A_{\vec{p}, G}$, where $r$ is the rank of the directed adjaceny matrix of $G$ over $\ZZ/p\ZZ$, when $G$ is bipartite. 
When considering this special case with a prime diagonal, experimental data suggests that Proposition~\ref{thm:bipartite_elementary_divisors} will apply to all graphs, not just bipartite graphs. 
In particular, whereas the maximum path length of the graph determines an upper bound for all elementary divisors, the rank of the adjacency matrix seems to be determining how the powers on $p$ are being distributed across the nonunit elementary divisors.
We end with the following conjecture which makes this observation precise.

\begin{conj}
	Let $G$ be an increasing DAG on $[n]$ and let $p$ be prime. Let $B \in \ZZ/p\ZZ^{n \times n}$ be given by
	\[B = \sum_{(i,j) \in G} E_{i,j} \in \ZZ/p\ZZ^{n\times n},\]
	and let $r$ be the rank of $B$ over $\ZZ/p\ZZ$. Then if $\alpha_i$ are the elementary divisors of $A_{\vec{p}, G}$, we have that $\alpha_{r} = 1$ and $\alpha_{r+1} > 1$.
\end{conj}

\bibliographystyle{plain}
\bibliography{ref.bib}

\end{document}